\documentclass[10pt, reqno, final]{amsart}

\usepackage[T1]{fontenc}
\usepackage[utf8]{inputenc} 
\usepackage{baskervald} 

\usepackage{microtype} 

\makeatletter
\def\MT@register@subst@font{\MT@exp@one@n\MT@in@clist\font@name\MT@font@list 
  \ifMT@inlist@\else\xdef\MT@font@list{\MT@font@list\font@name,}\fi}
\makeatother

\usepackage[letterpaper, showframe]{geometry}
\usepackage{amsrefs}

\usepackage{thmtools}
\declaretheoremstyle[%
  spaceabove=6pt,spacebelow=6pt,%
  headfont=\normalfont\bfseries,%
  bodyfont=\normalfont\itshape%
]{thmstyle}
\declaretheorem[numberwithin=section, name=Theorem, refname={theorem,theorems}, Refname={Theorem,Theorems}, style=thmstyle]{theorem}
\declaretheorem[sibling=theorem, name=Lemma, refname={lemma,lemmas}, Refname={Lemma,Lemmas}, style=thmstyle]{lemma}
\declaretheorem[sibling=lemma, name=Corollary, refname={corollary,corollaries}, Refname={Corollary,Corollaries}]{corollary}

\declaretheorem[numberwithin=section, style=definition, name=Definition, refname={definition,definitions}, Refname={Definition,Definitions}]{definition}

\def\pfunctionspace{\mathring{B}_2^{0,1}}

\def\chigammaspace{B_2^{1,1}}

\def\scontrolspace{B_2^2}
\def\sdiscretecontrolspace{b_2^2}
\def\gcontrolspace{B_2^1}
\def\gdiscretecontrolspace{b_2^1}
\def\fcontrolspace{L_2}
\def\fdiscretecontrolspace{{\l}_{2}}
\def\bcontrolspace{B_2^{1+\epsilon}}
\def\bdiscretecontrolspace{b_2}
\def\ccontrolspace{B_2^{1+\epsilon}}
\def\cdiscretecontrolspace{b_2}

\def\controlVars{s, g, f, b, c}
\def\controlVarsWithN{s^n, g^n, f^n, b^n, c^n}
\def\controlVarsWithNSub{s_n, g_n, f_n, b_n, c_n}
\def\controlVarsWithTilde{\tilde{s}, \tilde{g}, \tilde{f}, \tilde{b}, \tilde{c}}
\def\controlVarsWithStar{s_*, g_*, f_*, b_*, c_*}

\def\controlSpace{H}
\def\controlSpaceFull{\scontrolspace{} (0,T)\times{} \gcontrolspace{} (0,T)\times{}
  \fcontrolspace{} (D)\times{} \bcontrolspace{} (D) \times{} \ccontrolspace{} (D)}

\let\controlSpaceWeaklyConverge\controlSpaceFull%

\def\controlVarsStronglyConverge{s, g, b, c}
\def\controlVarsStronglyConvergeWithTilde{\tilde{s}, \tilde{g}, \tilde{b}, \tilde{c}}

\def\controlVarsStronglyConvergeWithN{s^n, g^n, b^n, c^n}
\def\controlSpaceStronglyConverge{B_2^1(0,T) \times{} L_2(0,T) \times{} L_2(D) \times{} L_2(D)}

\def\discreteControlVars{[s]_n, [g]_n, [f]_{nN}, [b]_{n}, [c]_{n}}

\def\discreteControlSet{V_R^n}

\newcommand{\D}[2]{\frac{\partial{} #1}{\partial{} #2}}

\newcommand{\bk}[1]{\left\{#1\right\}}
\DeclareMathOperator*{\esssup}{ess\,sup}
\newcommand{\ZZ}{\mathbf{Z}}
\newcommand{\norm}[1]{\left\Vert{} #1\right\Vert}
\newcommand{\lnorm}[1]{\left\vert{} #1\right\vert}

\def\probI{\mathcal{I}}
\def\J{\mathcal{J}}
\def\P{\mathcal{P}}
\def\Q{\mathcal{Q}}
\def\probIn{\mathcal{I}_n}
\def\fIn{I_n}
\DeclareMathOperator{\Ind}{\mathbf{1}}
\let\l\ell%

\begin{document}
\title{Optimal Control of Coefficients in Parabolic Free Boundary Problems Modeling Laser Ablation}%
\author{Ugur G.\ Abdulla}%
\author{Jonathan Goldfarb}%
\thanks{*Mathematical Sciences Department, Florida Institute of Technology, Florida, United States of America}%
%
%
%
%
\begin{abstract}{%
Inverse Stefan problem arising in modeling of laser ablation of biomedical tissues is analyzed, where information on the coefficients, heat flux on the fixed boundary, and density of heat sources are missing and must be found along with the temperature and free boundary.
Optimal control framework is employed, where the missing data and the free boundary are components of the control vector, and optimality criteria are based on the final moment measurement of the temperature and position of the free boundary.
Discretization by finite differences is pursued, and convergence of the discrete optimal control problems to the original problem is proven.
}%
\end{abstract}%
%
%
\maketitle%
\section{Introduction and Motivation}\label{sec:introduction}
Consider the general one-phase Stefan problem~\cites{friedman88,meirmanov92}: find the temperature function $u(x,t)$ and the free boundary $x=s(t)$ from the following conditions:
\begin{gather}
  Lu\equiv {(a u_{x})}_{x} + b u_{x} + c u - u_{t}=f-\D{p}{x},\quad \text{in}~\Omega=\bk{(x,t):0 < x < s(t),~0 < t \leq T}\label{eq:intro-pde}
  \\
  u(x,0)=\phi(x),\quad 0\leq x \leq s(0)=:s_{0}\label{eq:intro-iv}
  \\
  a(0,t)u_{x}(0,t)=g(t),\quad 0\leq t \leq T\label{eq:intro-bvl}
  \\
  a\big(s(t),t\big)u_{x}\big(s(t),t\big) + \gamma\big( s(t),t\big)s'(t)=\chi\big( s(t),t\big),\quad 0\leq t \leq T\label{eq:intro-stefan}
  \\
  u\big( s(t),t\big)=\mu(t), \quad 0\leq t \leq T,\label{eq:intro-bvr}
  \end{gather}
  where $a$, $b$, $c$, $f$, $p$, $\phi$, $g$, $\gamma$, $\chi$, $\mu$ are known functions $a(x,t)\geq a_{0}>0$, $s_{0}>0$.
In the context of heat conduction, $\gamma$ represents latent heat absorbed or released by the melting or freezing at the boundary, $\chi$ a heat source or sink on the boundary, $f$ and $p$ characterize the density of the sources, $\phi$ is the initial temperature, $g$ is the heat flux on the fixed boundary $x=0$, and $\mu$ is the phase transition temperature.
The coefficients $a$, $b$, and $c$ represent the diffusive, convective, and reactive properties, respectively, in the domain $\Omega$.

Assume now that some of the data is not available, or involves some measurement error.
For example, assume that the coefficients $b$ and $c$, heat flux $g(t)$ on the fixed boundary $x=0$ and the ``regular part'' of the density of heat sources, $f(x,t)$ are unknown and must be found along with the temperature $u(x,t)$ and the free boundary $x=s(t)$.
In this case, additional information is required; assume that this information is provided in the form of a measurement of temperature and the position of the free boundary at the final time $t = T$,
\begin{equation}
  u(x,T) = w(x),\quad
  0 \leq t \leq s(T)=:\bar{s}\label{eq:intro-meas-endtime}
\end{equation}

Under these conditions, we are required to solve an \textbf{Inverse Stefan Problem (ISP)}: find functions $u(x,t)$ and $s(t)$, the boundary heat flux $g(t)$, and the density of sources $f(x,t)$ satisfying conditions~\eqref{eq:intro-pde}--\eqref{eq:intro-meas-endtime}.

Motivation for this type of inverse problem arose, in particular, from the modeling of bioengineering problems on the laser ablation of biological tissues through a Stefan problem~\eqref{eq:intro-pde}--\eqref{eq:intro-meas-endtime}, where $s(t)$ is the ablation depth at the moment $t$.
The unknown parameters of the model such as $b$, $c$, $g$, and $f$ are very difficult to measure through experiments.
Lab experiments pursued on laser ablation of biological tissues allow the measure of final temperature distribution and final ablation depth; consequently, ISP must be solved for the identification of some, or all, of the unknown parameters $a$, $b$, $c$, $g$, $f$, etc.\@

Still another important motivation arises from the optimal control of the laser ablation process.
A typical control problem arises when unknown control parameters, such as the intensity of the laser source $f$, heat flux $g$ on the known boundary, and the coefficients $b$, $c$, must be chosen with the purpose of achieving a desired ablation depth and temperature distribution at the end of the time interval.

Research into inverse Stefan problems proceeded in two directions: inverse Stefan problems with given phase boundaries
in~\cites{bell81,budak72,budak73,budak74,cannon64,carasso82,ewing79,ewing79a,goldman97,hoffman81,sherman71}, or inverse problems with unknown phase boundaries in~\cites{vasilev69,baumeister80,fasano77,goldman97,hoffman82,hoffman86,jochum80,jochum80a,knabner83,lurye75,niezgodka79,nochetto87,primicerio82,sagues82,talenti82, yurii80}.

In~\cite{abdulla13},~\cite{abdulla15}, a new variational formulation of ISP was
introduced and existence of a solution as well as convergence of the method of finite differences was proven.
Fr\'echet differentiability in Besov spaces in the new variational formulation was proven in~\cite{abdulla17},~\cite{abdulla16}.
The goal of this project is to extend the results of~\cite{abdulla15} on the existence of a solution and convergence
of the method of finite differences to the identification of $f$, $b$, and $c$.

The structure of the remainder of the paper is as follows: in Section~\ref{sec:notation} we define all
the functional spaces.
Section~\ref{sec:isp-optcont-formulation} formulates the optimal control
problem; the discrete optimal control problem is formulated in
Section~\ref{sec:fullydisc-optimal-control-prob}.
The main results are formulated in Section~\ref{sec:main-result-formulation}.
In Section~\ref{sec:preliminary-results-disc} preliminary results are proven.
The proofs of the main results are elaborated in Section~\ref{sec:main-results}.
Finally, conclusions are presented in Section~\ref{sec:conclusions}.

\subsection{Notation}\label{sec:notation}
Let $U$ be open subset of the real line $\Re$.
\begin{itemize}
  \item The Sobolev-Besov space $B_{2}^{k}(U)$, for $k = 1, 2, \ldots$ is the Banach space of $L_{2}(U)$ functions whose weak derivatives up to order $k$ exist and are in $L_{2}(U)$.
  The norm in $B_{2}^{k}(U)$ is
  \[
    \norm{u}_{B_2^k(U)}^2
    := \sum_{i=0}^{k} \norm{\frac{d^k u}{dx^k}}_{L_2(U)}^2
  \]
  \item For $\l \not\in \ZZ_+$, $B_2^{\l}(U)$ is the Banach space of measurable functions with finite norm
  \[
    \norm{u}_{B_2^{\l}(U)}
    :=
    \norm{u}_{B_2^{(\l)}(U)}
    + [u]_{B_2^{\l}(U)}
  ~\text{where}~
    [u]_{B_2^\l(U)}^2
    :=\int_{U} \int_{U} \frac{
    \lnorm{
      \D{^{[\l]} u(x)}{x^{[\l]}}
      - \D{^{[\l]} u(y)}{x^{[\l]}}
    }^2
    }{
    \lnorm{x-y}^{1 + 2(\l - [\l])}
    } \,dx \,dy
  \]
\item Let $\l_1,\l_2>0$ and $D = U \times (0,T)$.
The Besov space $B_{2}^{\l_1, \l_2}(D)$ is defined as the closure of the set of smooth functions under the norm
\[
    \norm{u}_{B_{2}^{\l_1, \l_2}(D)}
    := \left(\int_0^T \norm{u(x,t)}_{B_2^{\l_1}(U)}^2 \,dt\right)^{1/2}
    + \left(\int_{U} \norm{u(x,t)}_{B_2^{\l_2}(0,T)}^2 \,dx\right)^{1/2}.
\]
When $\l_{1} = \l_{2} \equiv \l$, the corresponding Besov space is denoted by $B_{2}^{\l}(D)$.
$\mathring{B}_2^{\l_{1},\l_{2}}(D)$ denotes the closure of the set of smooth functions with compact support with respect to $x$ in $U$ under the $B_{2}^{\l_{1},\l_{2}}$-norm.
  \item $V_{2}(\Omega)$ is the subspace of $B_{2}^{1,0}(\Omega)$ for which the norm
  \[
    \norm{u}_{V_{2}(\Omega)}^2
    = \esssup_{0\leq t \leq T} \norm{u(\cdot, t)}_{L_{2}\big(0,s(t)
    \big)}^{2}
    + \norm{\D{u}{x}}_{L_{2}(\Omega)}^{2}
  \]
  \item $V_{2}^{1,0}(\Omega)$ is the completion of $B_{2}^{1,1}(\Omega)$ in the $V_{2}(\Omega)$ norm.
  $V_{2}^{1,0}(\Omega)$ is a Banach space with norm
  \[
    \norm{u}_{V_{2}^{1,0}(\Omega)}^2
    = \max_{0\leq t \leq T} \norm{u(\cdot, t)}_{L_{2}\big(0,s(t)\big)}^{2}
    + \norm{\D{u}{x}}_{L_{2}(\Omega)}^{2}
  \]
\end{itemize}
In the next section we describe the new variational formulation of this inverse problem.
\section{Variational Formulation of ISP}\label{sec:isp-optcont-formulation}
Fix any $R > 0, 0 < \epsilon \ll 1$ and nonnegative numbers $\beta_i, i=0,1,2$. Consider
\begin{gather}
  \J(v)
  = \beta_{0}\int_0^{s(T)}\lnorm{u(x,T)-w(x)}^2 \,dx
  + \beta_{1}\int_0^T\lnorm{u\big(s(t),t\big)-\mu(t)}^{2}\,dt
  + \beta_{2}|s(T)-\bar{s}|^{2}\label{eq:opt-functional}
\end{gather}
over the control set $V_{R}$:
\begin{gather}
  V_{R}=\Big{\{}
  v=(\controlVars{}) \in \controlSpace{}:
  ~0<\delta \leq s(t),
  ~s(0)=s_{0},
  ~s'(0)=0,
  \norm{v}_{\controlSpace} \leq R
  \Big{\}},\label{eq:opt-controlset}\nonumber
  \\
  \controlSpace{} = \controlSpaceFull{},\nonumber
  \\
  \norm{v}_{\controlSpace}
  :=
  \max\left(
  \norm{s}_{\scontrolspace(0,T)};
  \norm{g}_{\gcontrolspace(0,T)};
  \norm{f}_{\fcontrolspace(D)};
  \norm{b}_{\bcontrolspace(D)};
  \norm{c}_{\ccontrolspace(D)}
  \right)
\end{gather}
where $D$ is defined by
\[
  D := \bk{(x,t) : 0\leq x\leq \l,~ 0\leq t\leq T},
\]
where $\l=\l(R)>0$ is chosen such that for any control $v\in V_R$, its component
$s$ satisfies $s(t)\leq \l$.
Existence of appropriate $\l$ follows from Morrey's
inequality~\cites{ladyzhenskaya68,besov79}.
Let the function $f \in \fcontrolspace(D)$ be extended to $\fcontrolspace(\Re^2)$
by zero.
For given $v \in V_R$ the state vector $u = u(x,t;v)$ solves~\eqref{eq:intro-pde}--\eqref{eq:intro-stefan}.
Call the previous minimization problem by Problem $\probI$.

\begin{definition}\label{defn:intro-weak-w211}
  $u \in B_{2}^{1,1}(\Omega)$ is called a weak solution of the problem~\eqref{eq:intro-pde}--\eqref{eq:intro-stefan} if $u(x,0)=\phi(x) \in B_{2}^{1}(0,s_{0})$ and
  \begin{gather}
    0=\int_{0}^{T}\int_{0}^{s(t)}\big[
      au_{x}\Phi_{x}
      - bu_{x}\Phi
      - cu\Phi
      + u_{t}\Phi
      + f\Phi
      + p \Phi_x
      \big]\, dx\,dt\nonumber
    \\
    +\int_{0}^{T}\left[
      {\gamma\big(s(t),t\big)s'(t)}
      - {\chi\big(s(t),t\big)}\right]\Phi(s(t),t)\,dt
    + \int_{0}^{T}g(t)\Phi(0,t)\,dt\label{eq:intro-weak-w211}
  \end{gather}
  for any $\Phi \in B_{2}^{1,1}(\Omega)$.
\end{definition}
\begin{definition}\label{defn:intro-weak-v2}
  $u \in V_{2}(\Omega)$ is called a weak solution of the problem~\eqref{eq:intro-pde}--\eqref{eq:intro-stefan} if $u(x,0)=\phi(x) \in B_{2}^{1}(0,s_{0})$ and
  \begin{gather}
    0
    =\int_{0}^{T}\int_{0}^{s(t)}\big[
      au_{x}\Phi_{x}
      - bu_{x}\Phi
      - cu\Phi
      + u_{t}\Phi
      + f\Phi
      + p \Phi_x
      \big]\, dx\,dt
    -\int_{0}^{s(0)}\phi(x)\Phi(x,0)\,dx\nonumber
    \\
    +\int_{0}^{T}g(t)\Phi(0,t)\, dt 
    +\int_{0}^{T}\left[
      {\gamma\big(s(t),t\big)s'(t)}-u\big( s(t),t\big)s'(t)-{\chi\big(s(t),t\big)}
      \right]\Phi(s(t),t) \, dt\label{eq:intro-weak-v2}
  \end{gather}
  for any $\Phi \in B_{2}^{1,1}(\Omega)$ with $\Phi(x,T) = 0$.
\end{definition}


\section{Discrete Optimal Control Problem}\label{sec:fullydisc-optimal-control-prob}

Let
\[
  \omega_{\tau}=\bk{ t_{j}=j  \tau,~j=0,1,\ldots,n}
\]
be a grid on $[0,T]$ and $\tau=\frac{T}{n}$.
We will use the standard notation for differences of a sequence $\bk{d_{i}}$,
\begin{equation}
  d_{k,\bar{t}}
  =\frac{d_{k}-d_{k-1}}{\tau}
  ,\quad d_{kt}
  =d_{k+1,\bar{t}}
  ,\quad d_{k,\bar{t}t}
  =\frac{d_{k+1, \bar{t}}-d_{k,\bar{t}}}{\tau}
\end{equation}
Let us now introduce the spatial grid.
Given a discrete boundary $[s]_n\in \Re^{n+1}$, let $(p_0,p_1,\ldots,p_n)$ be a
permutation of $(0,1,\ldots,n)$ according to the order
\(
  s_{p_0}\leq s_{p_1}\leq \cdots \leq s_{p_n}
\).
In particular, according to this permutation for arbitrary $k$ there exists a unique $j_k$ such that
\(
  s_k=s_{p_{j_k}}
  \).
  Furthermore, unless it is necessary in the context, we are going to write simply
$j$ instead of subscript $j_k$.
Let
\[
  \omega_{p_0}=\{ x_{i}: x_i = i h,~i=0,1,\ldots,m_0^{(n)}\}
\]
be a grid on $[0,s_{p_0}]$ and $h=\frac{s_{p_0}}{m_0^{(n)}}$.
Furthermore we assume that
\begin{equation}
  h = O(\sqrt{\tau}), \quad \text{as}~ \tau \to 0.\label{eq:htau}
\end{equation}
We continue construction of the spatial grid by induction.
Having constructed $\omega_{p_{k-1}}$ on $[0,s_{p_{k-1}}]$ we construct
\[
  \omega_{p_k}=\{ x_i:~i=0,1,\ldots, m_k^{(n)} \}
\]
on $[0,s_{p_{k}}]$, where $m_k^{(n)}\geq m_{k-1}^{n}$, and inequality is strict if and only if $s_{p_{k}}>s_{p_{k-1}}$; for $i\leq m_{k-1}^{(n)}$ points $x_i$ are the same as in grid $\omega_{p_{k-1}}$.
Finally, if $s_{p_{n}}<\l$, then we introduce a grid on $[s_{p_n},\l]$
\[
  \overline{\omega}=\{x_i: x_i=s_{p_n}+(i-m_n^{(n)}) \overline{h},
  ~i=m_n^{(n)},\ldots, N \}
\]
of stepsize order $h$, i.e.\ $\overline{h}=O(h)$ as $h \to 0$.
Furthermore we simplify the notation and write $m_k^{(n)}\equiv m_k$.
Let
\[
  h_i=x_{i+1}-x_i, ~i=0,1,\ldots,N-1;
\]
and denote the space grid on $[0,\l]$ by $\omega_h$ and set
\(
  \Delta = \max_{i=0,\ldots,N-1} h_i
\)
Assume that
\(
m_k \to +\infty
\),
as
\(n\to \infty{}\).
Introduce the Steklov averages
\[
  h_{k}=\frac{1}{\tau}\int_{t_{k-1}}^{t_{k}}h(t)\,dt,\quad
  w_i = \frac{1}{h_i} \int_{x_i}^{x_{i+1}} w(x) \,dx,\quad
  d_{ik}=\frac{1}{h_i \tau} \int_{x_i}^{x_{i+1}}\int_{t_{k-1}}^{t_k} d(x,t)\,dt\,dx,
\]
where $i=0,1,\ldots,N-1$; $k=1,\ldots,n$; and $d$ stands for any of the functions
$a$, $p$, or $f$; and $h$ stands for any of the functions $\nu$, $\mu$, $g$, etc.\@
Define
\begin{gather}
  \discreteControlSet{} = \Big{\{}
  [v]_n = (\discreteControlVars{}) \in \bar{H}:~
  0<\delta \leq s_k;
  ~\norm{[v]_n}_{\bar{H}} \leq R
  \Big{\}}\label{eq:fully-discrete-control-set}
\end{gather}
where
\begin{gather*}
  \bar{H}:=\Re^{n + 1} \times \Re^{n + 1} \times \Re^{nN} \times
  \Re^{n+1} \times \Re^{n+1}
  \\
  \norm{[v]_n}_{\bar{H}} := \max \left(
  \norm{[s]_{n}}_{\sdiscretecontrolspace};
  \norm{[g]_{n}}_{\gdiscretecontrolspace};
  \norm{[f]_{nN}}_{\fdiscretecontrolspace};
  \norm{[b]_{n}}_{\cdiscretecontrolspace};
  \norm{[c]_{n}}_{\cdiscretecontrolspace}
  \right),
\end{gather*}
and
\begin{gather*}
  \norm{[g]_{n}}_{\gdiscretecontrolspace}^{2}
  = \sum_{k=0}^{n-1}\tau g_{k}^{2}
  + \sum_{k=1}^{n}\tau g_{k,\bar{t}}^{2},\quad
  \norm{[s]_{n}}_{\sdiscretecontrolspace}^{2}
  = \norm{[s]_n}_{\gdiscretecontrolspace}^2
  + \sum_{k=1}^{n-1}\tau s_{k,\bar{t}t}^{2},
  \\
  \norm{[f]_{nN}}_{\fdiscretecontrolspace{}}^2
  = \sum_{k=1}^n \sum_{i=0}^{N-1} \tau h_i f_{ik}^2,
  \quad
  \norm{[d]_{n}}_{\cdiscretecontrolspace{}}^2 =
  \sum_{k=0}^n \lnorm{d_k}^2,
\end{gather*}
where $s_{k} \equiv s_0$ for $k \leq 1$ and $d$ stands for either of $b$, $c$.
Let \(\bk{\psi_k,~k=0,1,\ldots}\) be an orthonormal set in
\(\bcontrolspace(D)\), and for simplicity of notation, denote the inner product
on the Hilbert space \(\bcontrolspace(D)\) by \(\langle{}
\cdot, \cdot \rangle_{\bcontrolspace}\).
Introduce two mappings $\Q{}_n$ and $\P{}_n$ between continuous and discrete control sets: Define
\(
  \Q{}_n(v)
  \)
  for $v \in V_R$ by $s_k=s(t_k)$ for $k=2,\ldots,n$, $g_k=g(t_k)$ for $k=0,1,\ldots,n$, and
\begin{gather*}
  f_{ik}=\frac{1}{h_i \tau} \int_{t_{k-1}}^{t_k} \int_{x_i}^{x_{i+1}} f(x,t)
  \,dx \,dt,\quad k=0,\ldots,n,~i=0,\ldots,N,
\end{gather*}
and
\(
d_{k} = \langle{} d, \psi_k\rangle_{\bcontrolspace}
\) for \(k=0,1,\ldots{}\)
where $d$ stands for either of $b$ or $c$.
Define
\(
  \P{}_n([v]_n)=v^n=(\controlVarsWithN)\in \controlSpace
\)
for $[v]_n \in \discreteControlSet{}$ by
\begin{gather}\label{Eq:W:1:11}
  s^n(t)=
    s_{k-1}+\left(t-t_{k-1}-\frac{\tau}{2}\right)s_{k-1, \bar{t}}+\frac{1}{2}(t-t_{k-1})^2
    s_{k-1, \bar{t}t},~t_{k-1}\leq t \leq t_k,~ k=\overline{1,n}.
  \\
  g^n(t)=g_{k-1}+g_{k,\bar{t}}(t-t_{k-1}),~t_{k-1} \leq t \leq t_k,
  k=\overline{1,n},\nonumber
  \\
  f^n(x,t) = f_{ik}, ~x_i \leq x < x_{i+1}, ~t_{k-1} \leq t <
  t_k, ~i=\overline{0,N-1}, ~k=\overline{1,n}\nonumber
  \\
  d^{n}(x,t) = \sum_{k=0}^n d_k \psi_k(x,t)\nonumber
\end{gather}
where $d$ is either of $b$ or $c$.
Given $v=(\controlVars{}) \in V_R$, we define the Steklov averages of traces by
\begin{equation}\label{eq:stek-avg-trace}
\chi^{k}_s=\frac{1}{\tau} \int_{t_{k-1}}^{t_{k}}\chi(s(t),t) \,dt,\quad
(\gamma_s s')^k=\frac{1}{\tau} \int_{t_{k-1}}^{t_{k}}\gamma(s(t),t)s'(t) \,dt,\quad k=1,2,\ldots,n
\end{equation}
Given $[v]_{n} \in \discreteControlSet{}$ we define
Steklov averages $\chi^{k}_{s^n}$ and $(\gamma_{s^n} (s^n)')^k$
through~\eqref{eq:stek-avg-trace} with $s$ replaced by $s^n$.
The Steklov averages \(b_{ik}\) and \(c_{ik}\) are defined through
\begin{equation}\label{eq:stek-avg-cn-bn}
  b_{ik} = \frac{1}{h_i \tau} \int_{t_{k-1}}^{t_k} \int_{x_i}^{x_{i+1}} b^n(x,t)
  \,dx \,dt,\quad
  c_{ik} = \frac{1}{h_i \tau} \int_{t_{k-1}}^{t_k} \int_{x_i}^{x_{i+1}} c^n(x,t)
  \,dx \,dt
\end{equation}
Next we define a discrete state vector through discretization of the integral identity~\eqref{eq:intro-weak-w211}
\begin{definition}\label{discretestatevector}
  Given discrete control vector $[v]_n \in \discreteControlSet{}$, the vector function
  \[
    [u([v]_n)]_n=(u(0),u(1),\ldots,u(n)), ~u(k) = (u_0,\ldots,u_N)\in \Re^{N+1}, ~k=0,\ldots,n
  \]
  is called a discrete state vector if
  \begin{description}
  \item[(a)] First $m_0+1$ components of the vector $u(0)$ satisfy
    \(
      u_i(0)=\phi_i := \phi(x_i), ~i=0,1,\ldots,m_0;
      \)
  \item[(b)] 
    For arbitrary $k=1,\ldots,n$ first $m_j+1$ components of the vector $u(k)$ solve the following system of $m_j+1$ linear algebraic equations:
    \begin{gather}
      \Big[ a_{0k}+h b_{0k}-h^2c_{0k}+\frac{h^2}{\tau} \Big] u_0(k) - \Big[
      a_{0k}+hb_{0k} \Big] u_1(k)=\frac{h^2}{\tau}u_0(k-1)-h^2f_{0k}-hg^n_{k} - h p_{0k} ,
      \nonumber
      \\
      -a_{i-1,k}h_i u_{i-1}(k)+\Big[
      a_{i-1,k}h_i+a_{ik}h_{i-1}+b_{ik}h_i h_{i-1}-c_{ik}h_i^2h_{i-1}+\frac{h_i^2h_{i-1}}{\tau}
      \Big] u_i(k) \nonumber
      \\
      -\Big[ a_{ik}h_{i-1}+b_{ik}h_i h_{i-1} \Big] u_{i+1}(k) =
      -h_i^2h_{i-1}f_{ik}
      +h_i h_{i-1} p_{ik,\bar{x}}
      +\frac{h_i^2 h_{i-1}}{\tau}u_i(k-1), ~i=1,\ldots,m_j-1
      \nonumber
      \\
      -a_{m_j-1,k} u_{m_j-1}(k)+a_{m_j-1,k} u_{m_j}(k)=-h_{m_j-1} \Big[  (\gamma_{s^n} (s^n)')^k-\chi^{k}_{s^n} \Big].\label{alma}
    \end{gather}
    \item[(c)] For arbitrary $k=0,1,\ldots,n$, the remaining components of $u(k)$ are calculated as
    \(
      u_i(k)= \hat{u}(x_i;k)
      \)
      for \( m_j\leq i \leq N \)
    where $\hat{u}(x;k) \in B_2^1(0,\l)$ is a piecewise linear interpolation of $\{u_i(k): ~i=0,\ldots,m_j \}$, that is to say
    \[
      \hat{u}(x;k)=u_i(k)+u_{ix}(k) (x-x_i), ~x_i\leq x\leq
      x_{i+1}, i=0,\ldots,m_j-1,
    \]
    iteratively continued for $0 \leq x < \infty$ as
    \begin{equation}\label{Eq:W:1:14}
      \hat{u}(x;k)=\hat{u}(2^n s_k-x;k), ~2^{n-1}s_k \leq x \leq 2^n s_k, n=1,2,\ldots
    \end{equation}
  \end{description}
\end{definition}
Note that no more than
\(
  n^*=1+\log_2\Big[ \frac{\l}{\delta}\Big]
\)
reflections are required to cover $[0,\l]$.
It should be mentioned that for any $k=1,2,\ldots,n$ system~\eqref{alma} is equivalent to the following summation identity
\begin{gather}
  \sum_{i=0}^{m_j-1}h_i \Big[
  a_{ik}u_{ix}(k)\eta_{ix}
  -b_{ik}u_{ix}(k)\eta_i
  -c_{ik}u_i(k)\eta_i
  +f_{ik}\eta_i
  +p_{ik}\eta_{ix}
  +u_{i\overline{t}}(k)\eta_i
  \Big] \nonumber
  \\
  + \Big[ (\gamma_{s^n} (s^n)')^k-\chi^{k}_{s^n} \Big]\eta_{m_j}
  +g^n_k \eta_0=0,\label{Eq:W:1:19}
\end{gather}
for arbitrary numbers $\eta_i$, $i=0,1,\ldots,m_j$.

Consider a discrete optimal control problem of minimization of the cost functional
\begin{equation}\label{Eq:W:1:15}
  \fIn([v]_n)
  =\beta_{0}\sum_{i=0}^{m_n-1} h_i \Big( u_i(n) - w_i \Big)^2
  +\beta_{1}\tau\sum_{k=1}^n \Big( u_{m_k}(k)-\mu_k \Big)^2
  +\beta_2 \lnorm{s_n - \bar{s}}^2
\end{equation}
on the set $\discreteControlSet{}$ subject to the state vector defined in Definition~\ref{discretestatevector}.
Furthermore, formulated discrete optimal control problem will be called Problem $\probIn$.
Throughout, we use piecewise constant and piecewise linear interpolations of the discrete state vector:
given discrete state vector $[u([v]_n)]_n=(u(0),u(1),\ldots,u(n))$, let
\begin{gather*}
  u^\tau(x,t)=\hat{u}(x;k), \quad \text{if}~ t_{k-1}<t\leq t_k, ~0\leq x \leq \l{},
  ~k=\overline{0,n},
  \\
  \hat{u}^\tau(x,t)=\hat{u}(x;k-1)+\hat{u}_{\overline{t}}(x;k)(t-t_{k-1}), \quad
  \text{if}~ t_{k-1}<t\leq t_k, ~0\leq x \leq \l{}, ~k=\overline{1,n},
  \\
  \hat{u}^\tau(x,t)= \hat{u}(x;n), \quad \text{if}~ t\ge T, ~0\leq x \leq \l{}.
  \\
  \tilde{u}^\tau(x,t)=u_i(k), \quad \text{if}~ t_{k-1}< t \leq t_k, ~x_i \leq x
  < x_{i+1}, ~k=\overline{1,n}, ~i=\overline{0,N-1}.
\end{gather*}
Standard notations for difference quotients of the discrete state vector are employed:
\[
  u_{ix}(k)=\frac{u_{i+1}(k)-u_i(k)}{h_i},
  ~u_{i\overline{t}}=\frac{u_i(k)-u_i(k-1)}{\tau}, \quad \text{etc.}
\]
Let $\phi^n$ be a piecewise constant approximation to $\phi$:
\[
  \phi^n(x) = \phi_i,~x_i < x \leq x_{i+1},~i=0,\ldots,N-1
\]
\section{Main Results}\label{sec:main-result-formulation}
Assume that the following conditions are satisfied
\begin{gather*}
  a(x,t)\geq a_0>0, a \in B_{\infty}^{1,0}(D)
  ~\text{with}~ M:=\norm{a}_{B_{\infty}^{1,0}(D)},
  ~\text{and}~
  \int_{-1}^T \esssup_{0 \leq x \leq \l} \lnorm{\D{a}{t}} \,dt < +\infty
  \\
  w \in L_2(0,\l),
  ~\chi,\gamma \in \chigammaspace(D),
  ~\phi \in B_2^1(0,s_0),
  ~\mu \in L_2(0,T),
  ~p \in \pfunctionspace(D_{\delta}).
\end{gather*}
where $D_{\delta} = (0,\delta) \times (0,T)$.
Note that the distributional derivative $\D{p}{x}$ is understood in the sense of measures.
Extend arbitrary $\mu \in L_2(0,T)$ to $L_2(\Re)$ by zero.
The main results of this work are the following:
\begin{theorem}[Existence of an Optimal Control]\label{thm:existence-opt-control}
  Problem $\probI$ has a solution.
  That is,
  \[
    V_{*}:=\bk{v \in {V_R}: \J(v) = J_{*}=:\inf_{v \in {V_R}} \J(v)} \neq \emptyset
  \]
\end{theorem}
\begin{theorem}\label{thm:convergence-discrete}
  $\probIn$ approximate the continuous problem $\probI$ \emph{with respect to functional} in the sense that
  \[
    \lim_{n\to\infty}I_{n}^{*} = J_{*},\quad \text{where}~I_{n}^{*} = \inf_{\discreteControlSet{}}\fIn,~\text{and}~J_*=\inf_{V_{R}}\J
  \]
  Moreover, the sequence $\probIn$ approximates $\probI$ \emph{with respect to control} in the sense that if $[u]_{n,\epsilon} \in \discreteControlSet{}$ is chosen such that
\[
  I_{n}^{*} \leq \fIn([v]_{n, \epsilon}) \leq I_{n}^{*} + \epsilon_{n},
  \quad \text{where}~\epsilon_{n}\downarrow 0
\]
then
the sequence $v^n=(\controlVarsWithN{}) = \P{}_n([v]_{n,\epsilon})$ converges to an element $v_*=(\controlVarsWithStar{}) \in V_*$ weakly in $\controlSpaceWeaklyConverge{}$.
In particular, $(\controlVarsStronglyConvergeWithN{})$ converge strongly in $\controlSpaceStronglyConverge{}$.
  Moreover, $s^n$ converges to $s_*$ uniformly on $[0,T]$.
  For any $\delta>0$, define
  \[
    \Omega' = \Omega \cap \bk{x<s(t)-\delta,~0<t<T}
  \]
  Then the piecewise linear interpolations $\hat{u}^\tau$ of the corresponding discrete state vectors $\big[[v]_{n,\epsilon}\big]_n$ converge to the solution $u(x,t;v_*) \in B_2^{1,1}(\Omega_*)$ of the Neumann problem~\eqref{eq:intro-pde}--\eqref{eq:intro-stefan} weakly in $B_2^{1,1}(\Omega')$.
\end{theorem}

\section{Preliminary Results}\label{sec:preliminary-results-disc}
\begin{lemma}\label{thm:pn-eq-maps-props:1}
  For arbitrary sufficiently small $\epsilon>0$, there exists $n_{\epsilon}$ such that
  \begin{gather}
    \Q{}_{n}(v) \in \discreteControlSet{},~\text{for all}~v \in V_{R-\epsilon},~n > n_{\epsilon}\label{eq:thm:pn-eq-maps-props:1-1}
    \\
    \P{}_{n}\big([v]_{n}\big) \in V_{R+\epsilon},~\text{for all}~[v]_n \in \discreteControlSet{},~n > n_{\epsilon}\label{eq:thm:pn-eq-maps-props:1-2}
  \end{gather}
\end{lemma}
\begin{proof}
  The first two components of either $\Q{}_n(v)$ for $v \in V_{R-\epsilon}$ or
  $\P{}_n\big([v]_n\big)$ for $[v]_n \in \discreteControlSet{}$ are estimated as
  in~\cite{abdulla15}*{Lem.\ 2.2}; all that remains is the estimation of the
  components corresponding to $f$, $b$, and $c$ in both.
  Since the components corresponding to $b$ and $c$ are in the same control set,
  we will give full details of for the $c$ component only, as those
  corresponding to $b$ are identical.

  Fix $v \in V_{R-\epsilon}$ and let $\big(\discreteControlVars\big) = \Q{}_n(v)$.
     By Cauchy-Bunyakovski-Schwarz (CBS) inequality,
     \begin{equation}
  \norm{[f]_{nN}}_{\fdiscretecontrolspace}^2
    \leq \sum_{k=1}^n \sum_{i=0}^{N-1}  \int_{x_i}^{x_{i+1}}\int_{t_{k-1}}^{t_k} \lnorm{f(x,t)}^2\,dt\,dx
      =\int_0^T \int_0^{\l} \lnorm{f(x,t)}^2 \,dx \,dt 
      = \norm{f}_{\fcontrolspace(D)}^2 
      \leq (R-\epsilon)^2\label{eq:Pn-Qn-fterm-est1}
\end{equation}
By Bessel's inequality,
\begin{gather}
\norm{[c]_n}_{\cdiscretecontrolspace}^{2}
= \sum_{k=0}^{n} \lnorm{c_{k}}^{2}
\leq \sum_{k=0}^{\infty} \lnorm{c_{k}}^{2}
\leq \norm{c}_{\ccontrolspace(D)}^{2} \leq (R-\epsilon)^2\label{eq:Pn-Qn-cterm-est1}
\end{gather}
By~\eqref{eq:Pn-Qn-fterm-est1},~\eqref{eq:Pn-Qn-cterm-est1}, and the proof of~\cite{abdulla15}*{Lem.\ 2.1}, it follows that
\[
\norm{\Q{}_n(v)}_{\discreteControlSet}^2 \leq R^2
\]
for \(\tau{}\) sufficiently small, which implies~\eqref{eq:thm:pn-eq-maps-props:1-1}.
Consider $[v]_n \in \discreteControlSet{}$ and let $\big(\controlVars\big) = \P{}_n([v]_n)$.
  Calculate
\begin{align}
  \norm{f}_{\fcontrolspace(D)}^2
    & = \int_0^T \int_0^{\l} \lnorm{f(x,t)}^2 \,dx \,dt
    = \sum_{k=1}^{n} \tau \sum_{i=0}^{N-1} h_i \lnorm{f_{ik}}^2 
    = \norm{[f]_n}_{\fdiscretecontrolspace(0,\l)}^2 
    \leq R^2\label{eq:Pn-Qn-fterm-est2}
\end{align}
By definition,
\begin{gather}
\norm{c^n}_{\ccontrolspace(D)}^{2}
= \left\langle c^{n}, c^{n} \right\rangle_{\ccontrolspace}
= \sum_{k=0}^{n} \sum_{j=0}^{n} c_{k} c_{j} \left\langle \psi_{k}, \psi_{j}\right\rangle_{\ccontrolspace}
= \norm{[c]_{n}}_{\cdiscretecontrolspace(D)}^{2}
\leq R^2\label{eq:Pn-Qn-cterm-est2}
\end{gather}
  By~\eqref{eq:Pn-Qn-fterm-est2},~\eqref{eq:Pn-Qn-cterm-est2}, and the proof of~\cite{abdulla15}*{Lem.\ 2.1}, it follows that
\[
\norm{\P{}_n([v]_{n})}_{\controlSpace{}}^2 \leq (R+\epsilon)^2
\]
for \(\tau{}\) sufficiently small, which implies~\eqref{eq:thm:pn-eq-maps-props:1-2}.
Lemma is proved.
\end{proof}
As in~\cite{abdulla13}, it follows from Theorem~\ref{thm:pn-eq-maps-props:1} that
\begin{corollary}\label{cor:pn-eq-maps-props:2}
  Let either $[v]_{n}\in \discreteControlSet{}$ or $[v]_{n}= \Q{}_{n}(v)$ for $v \in V_{R}$.
  Then for large $n$,
  \begin{equation}
    \lnorm{s_{k} - s_{k-1}} \leq C' \tau,~k = 1, 2, \ldots, n\label{eq:control-lipschitz-condition}
  \end{equation}
  where $C'$ is independent of $n$.
\end{corollary}
Note that for the step size $h_i$ we have one of the three possibilities: $h_i=h$, or $h_i=\overline{h}$, or $h_i \leq |s_{k}-s_{k-1}|$ for some $k$.
Hence, from~\eqref{eq:htau} and~\eqref{eq:control-lipschitz-condition}, it follows that
\begin{equation}\label{eq:htau-implication}
  \Delta = O(\sqrt{\tau}), \quad \text{as}~ \tau \rightarrow 0.
\end{equation}
Using Lemma~\ref{thm:pn-eq-maps-props:1}, we derive
\begin{corollary}\label{cor:cn-bn-unif-bdd}
  For a given discrete control vectors \([b]_n \in \bdiscretecontrolspace{}\), the
  coefficients \(\bk{b_{ik}}\) defined
  by~\eqref{eq:stek-avg-cn-bn} satisfy the estimate
  \begin{equation}
    \max_{ik} \lnorm{b_{ik}} \leq C\norm{[b]_n}_{\bdiscretecontrolspace}
  \end{equation}
  for \(C\) independent of \(n\) and \([b]_n\).
  In particular, \(\bk{b_{ik}}\) are uniformly bounded whenever
  \(\norm{[b]_n}_{\bdiscretecontrolspace{}}\) are bounded.
  Similarly, the coefficients \(\bk{c_{ik}}\) are uniformly bounded.
\end{corollary}
\begin{proof}
By embedding of $\ccontrolspace(D)$ in $L_{\infty}(D)$~\cites{besov79,nikolskii75,solonnikov64,solonnikov65},
\begin{gather*}
  \max_{ik} \lnorm{c_{ik}}
  = \max_{ik} \frac{1}{h_i \tau} \lnorm{\int_{t_{k-1}}^{t_k} \int_{x_i}^{x_{i+1}} c^{n}(x,t) \,dx \,dt}
  \leq \max_{ik} \sup_{t_{k-1} \leq t \leq t_k,~x_i \leq x \leq x_{i+1}} \lnorm{c^{n}(x,t)}
  \\
  = \norm{c^{n}}_{L_{\infty}(D)}
  \leq C \norm{c^{n}}_{\ccontrolspace(D)}
  \leq C \norm{[c]_{n}}_{\cdiscretecontrolspace(D)}\qedhere%
\end{gather*}
\end{proof}
\begin{lemma}\label{thm:disc-state-well-defined}
  For given $[v]_{n} \in \discreteControlSet{}$, the discrete state vector $\big[u\big( [v]_{n}\big) \big]_{n}$ exists and is unique for all sufficiently small $\tau>0$.
\end{lemma}
Lemma~\ref{thm:disc-state-well-defined} is established in nearly the same way as~\cite{abdulla15}*{Lem.\ 2.1}.
\section{Proof of Main Results}\label{sec:main-results}
\subsection{Energy Estimates and their Consequences}\label{sec:energy-est}
\begin{theorem}\label{thm:first-energy-est}
  For $\tau$ sufficiently small, and for any discrete control $[v]_{n}\in
    \discreteControlSet{}$, the corresponding discrete state vector satisfies the estimate
  \begin{gather}
    \max_{0\leq k \leq n} \sum_{i=0}^{N-1}
    h_i u_i^2(k)+\sum_{k=1}^{n}\tau\sum_{i=0}^{N-1}h_i u_{ix}^2(k)
    \leq C \Big (
    \norm{ \phi^n }_{L_2(0,s_0)}^2
    + \norm{ g^n }_{L_2(0,T)}^2
    + \norm{ f^n }_{L_2(D)}^2\nonumber
    \\
    + \norm{ \gamma(s^n(t),t)(s^n)'(t) }_{L_2(0,T)}^2
    + \norm{ \chi(s^n(t),t) }_{L_2(0,T)}^2
    +\sum_{k=1}^{n-1}\Ind_{+}(s_{k+1}-s_{k}) \sum_{i=m_j}^{m_{j_{k+1}}-1} h_i u_i^2(k) \Big ),\label{eq:thm:first-energy-est-2}
  \end{gather}
\end{theorem}
Theorem~\ref{thm:first-energy-est} is an extension of~\cite{abdulla15}*{Thm.\ 3.1}.
As in~\cite{abdulla13}{Thm. 3.4}, from Theorem~\ref{thm:first-energy-est} we have
\begin{theorem}\label{thm:first-energy-est-conseq}
  Let $[v]_{n}\in \discreteControlSet{}$ for $n=1,2,\ldots$ be a sequence of discrete controls
  with $\bk{\P_n\big([v]_{n}\big)}$ converging weakly in $\controlSpaceWeaklyConverge{}$ to an
  element $v=(\controlVars{})$ (and hence with $(\controlVarsStronglyConvergeWithN)$ converging strongly in $\controlSpaceStronglyConverge$).
  Then $\bk{u^{\tau}}$
  converges as $\tau \to 0$ weakly in $B_{2}^{1,0}(D)$ to a weak solution $u \in V_{2}^{1,0}(\Omega)$ of~\eqref{eq:intro-pde}--\eqref{eq:intro-stefan}.
  Moreover, $u$ satisfies the energy estimate
  \begin{gather}
    \norm{u}_{V_{2}^{1,0}(D)}^{2}
    \leq C\Big[\norm{\phi}_{L_{2}(0,s_{0})}^{2}
    + \sup_n \norm{f^n}_{L_{2}(D)}^{2}
    + \norm{p}_{L_{2}(D)}^{2}
    + \norm{\gamma}_{B_{2}^{1,0}(D)}^{2}
    + \norm{\chi}_{B_{2}^{1,0}(D)}^{2}
    + \norm{g}_{L_{2}(0,T)}^{2}\Big]\label{eq:thm:first-energy-est-conseq}
  \end{gather}
\end{theorem}
Equivalence of the piecewise constant interpolation $b_{ik}$ to $b^{n}(x,t)$ in $L_{2}(D)$ follows from application of CBS inequality.
From Theorem~\ref{thm:first-energy-est-conseq} in particular we have
\begin{corollary}
  For any $v = (\controlVars{}) \in {V_R}$, there exists a weak solution $u \in V_{2}^{1,0}(\Omega)$ of the Neumann problem~\eqref{eq:intro-pde}--\eqref{eq:intro-stefan} satisfying the energy estimate~\eqref{eq:thm:first-energy-est-conseq}.
\end{corollary}
Given any discrete control vector $[v]_n$ and the corresponding discrete state
vector $\big[u([v]_n)\big]_n$, define the constant continuation
$\big[\tilde{u}([v]_n)\big]_n$ by
$
  \tilde{u}_i(k) =
    u_i(k)$ for $0 \leq i \leq m_j$
    and $\tilde{u}_i(k) = u_{m_j}(k)$ for $m_j < i$
for $k=0,\ldots,n$.
\begin{theorem}[Second Energy Estimate]\label{thm:second-energy-est}
  For $\tau$ sufficiently small, and for any discrete control $[v]_{n}\in
    \discreteControlSet{}$, the modified discrete state vector $\big[
      \tilde{u}([v]_n)\big]$ satisfies
  \begin{gather}
    \max_{1 \leq k \leq n}\sum_{i=0}^{m_{j}-1} h_i u_{ix}^2(k)
    + \tau \sum_{k=1}^n \sum_{i=0}^{m_j-1} h_i \tilde{u}_{i\bar{t}}(k)^2
    + \tau^2 \sum_{k=1}^n \sum_{i=0}^{m_j-1} h_i \tilde{u}_{ix\bar{t}}^2(k)
    \leq C\bigg[
    \norm{\phi^n}_{L_2(0,s_0)}^2
    + \norm{\phi}_{B_2^1(0,s_0)}^2\nonumber
    \\
    + \norm{f^n}_{L_2(D)}^2
    + \norm{g^n}_{B_2^{1/4}(0,T)}^{2}
    + \norm{p}_{\pfunctionspace(D_{\delta})}^2
    + \norm{\gamma(s^n(t), t) \big( s^n\big)'(t)}_{B_2^{1/4}(0,T)}^{2}
    + \norm{\chi(s^n(t), t)}_{B_2^{1/4}(0,T)}^{2}
    \bigg]\label{eq:thm:second-energy-est}
  \end{gather}
\end{theorem}
\begin{proof}
  In~\eqref{Eq:W:1:19}, take $\eta = 2\tau \tilde{u}_{i\bar{t}}(k)$ to derive
  \begin{gather}
    \sum_{i=0}^{m_j-1}2 \tau h_i \Big[ a_{ik}u_{ix}(k) \tilde{u}_{ix\bar{t}}(k)
    -b_{ik}u_{ix}(k) \tilde{u}_{i\bar{t}}(k)
    -c_{ik}u_i(k) \tilde{u}_{i\bar{t}}(k)
    +f_{ik} \tilde{u}_{i\bar{t}}(k)
    +p_{ik} \tilde{u}_{ix\bar{t}}(k)\nonumber
    \\
    +u_{i\bar{t}}(k) \tilde{u}_{i\bar{t}}(k) \Big]
    + 2\tau\Big[ {\big(\gamma_{s^{n}}(s^{n})'\big)^{k}}-{\chi_{s^{n}}^{k}} \Big] \tilde{u}_{m_j\bar{t}}(k)
    +2\tau g^n_k \tilde{u}_{0\bar{t}}(k)=0\label{eq:second-energy-est-disc-state}
  \end{gather}
  Arguing as in~\cite{abdulla15}*{Thm.\ 3.3} (in particular, Eq.\ 3.46), it follows that
\begin{gather}
    \sum_{i=0}^{m_{j_q}-1} h_i  \tilde{u}_{ix}^2(q)
    + \tau^2 \sum_{k=1}^q \sum_{i=0}^{m_j-1} h_i \tilde{u}_{ix\bar{t}}^2(k)
    + \tau \sum_{k=1}^q\sum_{i=0}^{m_j-1} h_i\tilde{u}_{i\bar{t}}^2(k)
    \leq C\sum_{i=0}^{m_{j_0}-1} h_i \phi_{ix}^2\nonumber
    \\
    + C \sum_{k=1}^q \sum_{i=0}^{m_j-1}\tau h_i \tilde{u}_{ix}^2(k)
    + C \max_{1\leq k \leq q} \sum_{i=0}^{m_j-1} h_i \tilde{u}_i^2(k)
    + C \sum_{k=1}^q \sum_{i=0}^{m_j-1}\tau h_i f_{ik}^2
    - \sum_{k=1}^q \sum_{i=0}^{m_j-1}\tau h_i p_{ik} \tilde{u}_{ix\bar{t}}(k)\nonumber
    \\
    - 2 \tau \sum_{k=1}^q \Big[ {\big(\gamma_{s^{n}}(s^{n})'\big)^{k}} - {\chi_{s^{n}}^{k}} \Big] \tilde{u}_{m_j\bar{t}}(k)
    - 2 \tau \sum_{k=1}^q g^n_k \tilde{u}_{0\bar{t}}(k)\label{eq:second-energy-est-before-change-p}
  \end{gather}
  for any $1 \leq q \leq n$.
  The second and third terms on the right-hand side of~\eqref{eq:second-energy-est-before-change-p} can be estimated using the first energy estimate;
For the term containing $p_{ik}$, we apply summation by parts; by virtue of the
compact support of $p$ with respect to $x$ in $(0,\delta)$, there exists
$i_{\delta}$ with $i_{\delta} < m_{j_k}-1$ for all $k$ such that $p_{ik} \equiv 0$
for $i>i_\delta$, and hence
\begin{gather}
  \sum_{k=1}^q \sum_{i=0}^{m_j-1}\tau h_i p_{ik} \tilde{u}_{ix\bar{t}}(k)
  = \sum_{i=0}^{i_{\delta}} h_i p_{iq} \tilde{u}_{ix}(q)
  - \sum_{i=0}^{i_{\delta}} h_i p_{i,1} \phi_{ix}
  - \sum_{k=1}^{q-1} \sum_{i=0}^{i_{\delta}} \tau h_i p_{i,k+1,\bar{t}} \tilde{u}_{ix}(k)\label{eq:second-energy-est-p-manip}
\end{gather}
Therefore,
from~\eqref{eq:second-energy-est-before-change-p},~\eqref{eq:second-energy-est-p-manip}, Corollary~\ref{cor:cn-bn-unif-bdd}, and Cauchy inequality with $\epsilon$, it follows that
\begin{gather}
  \frac{a_0}{2} \sum_{i=0}^{m_{j_q}-1} h_i  \tilde{u}_{ix}^2(p)
  + \tau^2 a_0 \sum_{k=1}^q \sum_{i=0}^{m_j-1} h_i \tilde{u}_{ix\bar{t}}^2(k)
  + \frac{\tau}{2} \sum_{k=1}^q\sum_{i=0}^{m_j-1} h_i\tilde{u}_{i\bar{t}}^2(k)
  \leq C \bigg{\{}\sum_{i=0}^{m_{j_0}-1} h_i \phi_{ix}^2\nonumber
  \\
  + \sum_{k=1}^q \sum_{i=0}^{m_j-1}\tau h_i \tilde{u}_{ix}^2(k)
  +  \max_{1\leq k \leq p} \sum_{i=0}^{m_j-1} h_i \tilde{u}_i^2(k)
  + \sum_{k=1}^q \sum_{i=0}^{m_j-1}\tau h_i f_{ik}^2
  + \sum_{i=0}^{i_{\delta}} h_i p_{iq}^2\nonumber
  \\
  + \sum_{i=0}^{i_{\delta}} h_i p_{i1}^2
  + \sum_{k=1}^{q-1} \sum_{i=0}^{i_{\delta}} \tau h_i p_{i,k+1,\bar{t}}^2
  + \tau\sum_{k=1}^q \Big[ {\big(\gamma_{s^{n}}(s^{n})'\big)^{k}} - {\chi_{s^{n}}^{k}} \Big] \tilde{u}_{m_j\bar{t}}(k)
  + \tau \sum_{k=1}^q g^n_k \tilde{u}_{0\bar{t}}(k)\bigg{\}}\label{eq:second-energy-est-after-change-p-3}
\end{gather}
for some $C$ independent of $\tau$.
By CBS inequality and Fubini's theorem we have
\begin{equation}
  \tau \sum_{k=1}^{m-1} \sum_{i=0}^{m_j-1} h_i p_{i,k+1,t}^2
    \leq \frac{1}{\tau^2} \sum_{k=1}^{m-1} \int_0^{s_k} \int_{t_{k-1}}^{t_{k}}\lnorm{p(x,t+\tau)  -  p(x,t)}^2\,dt\, dx
  \leq \norm{p_t}_{L_2(D)}^2\label{eq:p-tbar-term-est-final}
\end{equation}

By CBS inequality and Sobolev embedding theorem~\cites{nikolskii75, besov79}
\begin{gather}
  \sum_{i=0}^{m_j-1} h_i p_{ik}^2
  = \sum_{i=0}^{m_j-1} \frac{1}{h_i \tau^2}
  \left(\int_{x_i}^{x_{i+1}}\int_{t_{k-1}}^{t_k}p(x,t)\,dt \,dx\right)^2
  \leq \frac{1}{\tau} \int_{t_{k-1}}^{t_k} \int_{0}^{\delta}p^2(x,t)  \,dx \,dt
  \leq C \norm{p}_{\pfunctionspace(D_{\delta})}^2\label{eq:p-steklovtrace-est}
\end{gather}
Having~\eqref{eq:p-steklovtrace-est} and~\eqref{eq:p-tbar-term-est-final}, from~\eqref{eq:second-energy-est-after-change-p-3} it follows that
\begin{gather}
  \frac{a_0}{2} \sum_{i=0}^{m_{j_q}-1} h_i  \tilde{u}_{ix}^2(q)
  + \tau^2 a_0 \sum_{k=1}^q \sum_{i=0}^{m_j-1} h_i \tilde{u}_{ix\bar{t}}^2(k)
  + \frac{\tau}{2} \sum_{k=1}^q\sum_{i=0}^{m_j-1} h_i\tilde{u}_{i\bar{t}}^2(k)
  \leq C \bigg{\{}\sum_{i=0}^{m_{j_0}-1} h_i \phi_{ix}^2\nonumber
  \\
  + \sum_{k=1}^q \sum_{i=0}^{m_j-1}\tau h_i \tilde{u}_{ix}^2(k)
  +  \max_{1\leq k \leq q} \sum_{i=0}^{m_j-1} h_i \tilde{u}_i^2(k)
  + \sum_{k=1}^q \sum_{i=0}^{m_j-1}\tau h_i f_{ik}^2
  + \norm{p}_{\pfunctionspace(D_{\delta})}^2\nonumber
  \\
  + \tau\sum_{k=1}^q \Big[ {\big(\gamma_{s^{n}}(s^{n})'\big)^{k}} - {\chi_{s^{n}}^{k}} \Big] \tilde{u}_{m_j\bar{t}}(k)
  + \tau \sum_{k=1}^q g^n_k \tilde{u}_{0\bar{t}}(k)\bigg{\}}\label{eq:second-energy-est-after-change-p-5}
\end{gather}
  Since this inequality holds for all $1 \leq q \leq n$, it follows that
  \begin{gather}
    \frac{a_0}{2} \max_{1 \leq k \leq n}\sum_{i=0}^{m_{j_k}-1} h_i  \tilde{u}_{ix}^2(k)
    + \tau^2 a_0 \sum_{k=1}^n \sum_{i=0}^{m_j-1} h_i \tilde{u}_{ix\bar{t}}^2(k)
    + \frac{\tau}{2} \sum_{k=1}^n\sum_{i=0}^{m_j-1} h_i\tilde{u}_{i\bar{t}}^2(k)
    \leq C \bigg{\{}\sum_{i=0}^{m_{j_0}-1} h_i \phi_{ix}^2\nonumber
    \\
    + \sum_{k=1}^n \sum_{i=0}^{m_j-1}\tau h_i \tilde{u}_{ix}^2(k)
    +  \max_{1\leq k \leq n} \sum_{i=0}^{m_j-1} h_i \tilde{u}_i^2(k)
    + \norm{f}_{L_2(D)}^2
    + \norm{p}_{\pfunctionspace(D_{\delta})}^2\nonumber
    \\
    + \tau\sum_{k=1}^n \Big[ {\big(\gamma_{s^{n}}(s^{n})'\big)^{k}} - {\chi_{s^{n}}^{k}} \Big] \tilde{u}_{m_j\bar{t}}(k)
    + \tau \sum_{k=1}^n g^n_k \tilde{u}_{0\bar{t}}(k)\bigg{\}}\label{eq:second-energy-est-lemma-prefinal}
  \end{gather}
  The boundary terms containing $u_{\bar{t}}$ present another challenge.
  The proof of the corresponding energy estimate in~\cite{abdulla13} gives the idea to use inverse embedding of Sobolev spaces.

  If $\gamma, \chi \in B_2^{1,1}(D)$ and $[v]_n = (\discreteControlVars{}) \in
    \discreteControlSet$, then for $n$ large enough,
    \(
    \P_n([v]_n) \in V_{R+1}
    \)
  by Theorem~\ref{thm:pn-eq-maps-props:1},
  and hence the traces of $\chi$ and $\gamma \cdot (s^n)'$ on the curves $x=s^n(t)$ are in $B_2^{1/4}(0,T)$~\cites{nikolskii75, besov79} and
  \begin{gather}
    \norm{{\gamma\big(s^{n}(t),t\big)(s^{n})'(t)}}_{B_2^{1/4}(0,T)} \leq C \norm{\gamma}_{B_2^{1,1}(D)},\quad
    \norm{{\chi\big(s^{n}(t),t\big)}}_{B_2^{1/4}(0,T)} \leq C \norm{\chi}_{B_2^{1,1}(D)}\label{eq:second-energy-est-gamma-chi-bound}
  \end{gather}
  Let $\Psi(x,t) \in B_2^{2,1}(D)$ be a solution of the heat equation satisfying
  \begin{gather}
    \Psi(x,0)=\phi(x),~\text{for}~x \in [0,s_0],
    \quad
    a(0,t)\Psi_x(0,t) = g^n(t),~\text{for a.e.}~t \in [0,T],\nonumber
    \\
    a(s^n,t)\Psi_x(s^n(t),t) = {\chi\big(s^{n}(t),t\big)} - {\gamma\big(s^{n}(t),t\big)(s^{n})'(t)},
    ~\text{for a.e.} ~ t \in [0,T]\nonumber
  \end{gather}
  and
  \begin{gather}
    \norm{\Psi}_{B_2^{2,1}(D)}
    \leq C \Big[
    \norm{g^n}_{B_2^{1/4}(0,T)}
    + \norm{\phi}_{B_2^1(0,s_0)}
    + \norm{{\chi\big(s^{n}(t),t\big)} - {\gamma\big(s^{n}(t),t\big)(s^{n})'(t)}}_{B_2^{1/4}(0,T)}
    \Big]\label{eq:second-energy-est-w-bound}
  \end{gather}
  Existence of such $\Psi$ follows from e.g.~\cite{ladyzhenskaya68}*{Ch.\ 3, Thm.\ 6.1}.
  Then replacing $u$, $s$ and $g$ with $u-\Psi$, $s^n$, and $g^n$ in~\eqref{eq:second-energy-est-lemma-prefinal} with $f(x)$ replaced by
  $f(x) - L\Psi(x) \in L_2(D)$, we derive
  \begin{gather}
    \frac{a_0}{2} \max_{1 \leq k \leq n} \sum_{i=0}^{m_{j_k}-1} h_i  \tilde{u}_{ix}^2(k)
    + \tau^2 a_0 \sum_{k=1}^n \sum_{i=0}^{m_j-1} h_i \tilde{u}_{ix\bar{t}}^2(k)
    + \frac{\tau}{2} \sum_{k=1}^n\sum_{i=0}^{m_j-1} h_i\tilde{u}_{i\bar{t}}^2(k)
    \leq C \bigg{\{}\sum_{i=0}^{m_{j_0}-1} h_i \phi_{ix}^2\nonumber
    \\
    + \sum_{k=1}^n \sum_{i=0}^{m_j-1}\tau h_i \tilde{u}_{ix}^2(k)
    +  \max_{1\leq k \leq n} \sum_{i=0}^{m_j-1} h_i \tilde{u}_i^2(k)
    + \norm{f}_{L_2(D)}^2
    + \norm{L\Psi}_{L_2(D)}^2
    + \norm{p}_{\pfunctionspace(D_{\delta})}^2\bigg{\}}\label{eq:second-energy-est-v}
  \end{gather}
  By the first energy estimate~\eqref{eq:thm:first-energy-est-2}, along with~\eqref{eq:second-energy-est-w-bound}, and~\eqref{eq:second-energy-est-gamma-chi-bound},
  from~\eqref{eq:second-energy-est-v} it follows that for $\tau$ sufficiently small, $u$ satisfies
  \begin{gather}
    \frac{a_0}{2} \max_{1 \leq k \leq n} \sum_{i=0}^{m_{j_k}-1} h_i  \tilde{u}_{ix}^2(k)
    + \tau^2 a_0 \sum_{k=1}^n \sum_{i=0}^{m_j-1} h_i \tilde{u}_{ix\bar{t}}^2(k)
    + \frac{\tau}{2} \sum_{k=1}^n \sum_{i=0}^{m_j-1} h_i\tilde{u}_{i\bar{t}}^2(k)
    \leq C \bigg{\{}\norm{\phi}_{B_2^1(0,s_0)}\nonumber
    \\
    + \norm{\phi^n}_{L_2(0,s_0)}^2
    + \norm{f}_{L_2(D)}^2
    + \norm{g^n}_{B_2^{1/4}(0,T)}
    + \norm{\phi}_{B_2^1(0,s_0)}
    + \norm{p}_{\pfunctionspace(D_{\delta})}^2\nonumber
    \\
    + \norm{{\chi\big(s^{n}(t),t\big)} - {\gamma\big(s^{n}(t),t\big)(s^{n})'(t)}}_{B_2^{1/4}(0,T)}
    + \sum_{k=1}^{n-1}\Ind_{+}(s_{k+1} - s_{k}) \sum_{i=m_j}^{m_{j_{k+1}}-1}h_i
    u_i^{2}(k)\bigg{\}}\label{eq:second-energy-est-final}
  \end{gather}
  where $C$ independent of $\tau$ has been used to absorb the constants on the
  left-hand side, and $\tau$ is sufficiently small as in the hypotheses of
  Theorem~\ref{thm:first-energy-est}, which
  implies~\eqref{eq:thm:second-energy-est}.
\end{proof}
As in~\cite{abdulla15}{Thm.\ 3.4}, from Theorem~\ref{thm:second-energy-est} we have
\begin{theorem}\label{lem:second-energy-est-conseq}
  Let $[v]_{n} \in \discreteControlSet{}$ for $n = 1, 2, \ldots$ be a sequence of discrete
  controls with $\bk{\P_n\big( [v]_{n}\big)}$ converging weakly to an element $v
    = (\controlVars{})$ in $\controlSpace{}$ (with $(\controlVarsStronglyConvergeWithN)$ converging strongly in $\controlSpaceStronglyConverge$ to $(\controlVarsStronglyConverge)$) and, for any $\delta>0$, define
  \[
    \Omega' = \Omega \cap \bk{x<s(t)-\delta,~0<t<T}.
  \]
  Then $\bk{\hat{u}^{\tau}(x,t;v_n)}$ converges as $\tau \to 0$ weakly in $B_{2}^{1,1}(\Omega')$ to a weak solution $u \in B_{2}^{1,1}(\Omega)$ of~\eqref{eq:intro-pde}--\eqref{eq:intro-stefan}.
  Moreover, $u$ satisfies the energy estimate
  \begin{equation}
    \norm{u}_{B_{2}^{1,1}(\Omega)}^{2}
      \leq C\Big[
    \norm{\phi}_{B_{2}^{1}(0,s_{0})}^{2}
    + \sup_n\norm{f^n}_{L_{2}(D)}^{2}
    + \norm{p}_{B_2^{0,1}(D)}^2
    + \norm{\gamma}_{\chigammaspace(D)}^{2}
    + \norm{\chi}_{\chigammaspace(D)}^{2}
    + \norm{g}_{B_{2}^{1/4}(0,T)}^{2}
    \Big]\label{eq:thm:second-energy-est-conseq}
  \end{equation}
\end{theorem}
Theorem~\ref{lem:second-energy-est-conseq} implies the following
\begin{corollary}\label{cor:second-energy-est-conseq-2}
  For any $v \in {V_R}$, there exists a weak solution $u \in B_{2}^{1,1}(\Omega)$ of the Neumann problem~\eqref{eq:intro-pde}--\eqref{eq:intro-stefan} satisfying the energy estimate~\eqref{eq:thm:second-energy-est-conseq}.
  By Sobolev extension theorem, $u$ may be extended to a $B_2^{1,1}(D)$ function with norm preservation, so it satisfies the energy estimate
  \begin{equation*}
      \norm{u}_{B_{2}^{1,1}(D)}^{2}
        \leq C\Big[
      \norm{\phi}_{B_{2}^{1}(0,s_{0})}^{2}
      + \norm{f}_{L_{2}(D)}^{2}
      + \norm{p}_{\pfunctionspace(D_{\delta})}^2
      + \norm{\gamma}_{\chigammaspace(D)}^{2}
      + \norm{\chi}_{\chigammaspace(D)}^{2}
      + \norm{g}_{B_{2}^{1/4}(0,T)}^{2}
      \Big]
  \end{equation*}
\end{corollary}
\subsection{Existence and Convergence Results}\label{sec:existence-and-convergence}
\begin{proof}[Proof of Theorem~\ref{thm:existence-opt-control}]
  Let $\bk{v_{n}}\in {V_R}$ be a minimizing sequence for $\J$.
  Since ${V_R}$ is bounded in the Hilbert space $\controlSpace$, $v_{n}=(\controlVarsWithNSub)$ is weakly precompact in $\controlSpaceWeaklyConverge{}$.
  Assume that $v_{n}\to v=(\controlVars{})\in {V_R}$ weakly in $\controlSpaceWeaklyConverge{}$,
  and hence $(\controlVarsStronglyConverge{})$ converge strongly in
  $\controlSpaceStronglyConverge{}$.
  Let $u_{n},u \in B_{2}^{1,1}(D)$ be the corresponding solutions to the Neumann problem~\eqref{eq:intro-pde}--\eqref{eq:intro-stefan} in $B_{2}^{1,1}(\Omega_{n})$ and $B_{2}^{1,1}(\Omega)$, respectively, where
  \[
    \Omega_{n}=\bk{(x,t):0<x<s_{n}(t),~0<t<T}.
  \]
  $u_{n}$ and $u$ satisfy the
  estimate~\eqref{eq:thm:second-energy-est-conseq} with $(g_{n}, f_n)$ and
  $(g, f)$ respectively.
  Since $v_{n}\in V_{R}$, $u_{n}$ is in fact uniformly bounded in
  $B_{2}^{1,1}(D)$.
  Considering the sequence
  \(
    \Delta u = \Delta u_{n}= u_{n}-u,
  \)
  from Lemma~\ref{lem:second-energy-est-conseq} we have the rough estimate
  $\norm{\Delta u}_{B_{2}^{1,1}(D)} \leq C$ uniformly with respect to $n$.
  Therefore, $\bk{\Delta u}$ is weakly precompact in $B_{2}^{1,1}(D)$.

  Without loss of generality, assume that $u_{n}-u$ converges weakly in $B_{2}^{1,1}(D)$ to an element $w \in B_{2}^{1,1}(D)$.
  Assume temporarily that the fixed test function $\Phi \in C^1(\bar{D})$.
  Subtracting the integral identities satisfied by $u_{n}$ and $u$,
  we see that $\Delta u = u_{n}-u$ satisfies
  \begin{gather}
    0=\int_{0}^{T}\int_{0}^{s(t)}\big[ a{\Delta u}_{x}\Phi_{x}-b {\Delta u}_{x}\Phi -
    c {\Delta u}\Phi + {\Delta u}_{t}\Phi\big]\, dx\,dt
    + I_1 + I_2 + I_3 + I_4 + I_5,\label{eq:identity-for-limitpt-w}
    \intertext{where}
    I_1 := \int_{0}^{T}\int_{0}^{s(t)}\big[-\left( b_n - b \right) u_{n,x}\Phi -
      \left(c_n - c\right) u_{n}\Phi + \left(f_n - f\right) \Phi \big]\, dx\,dt\nonumber
    \\
    I_2 := -\int_{0}^{T}\int_{s_{n}(t)}^{s(t)}\big[ au_{n,x}\Phi_{x}-b_n u_{n,x}\Phi -
      c_n u_{n}\Phi + u_{n,t}\Phi + f_n \Phi + p \Phi_x\big]\, dx\,dt\nonumber
    \\
    I_3 := \int_{0}^{T}\left[ \gamma\big(s_{n}(t),t\big)s_{n}'(t)-\chi\big(
    s_{n}(t),t\big)\right]\left(  \Phi(s_{n}(t),t)-\Phi(s(t),t)\right)\,dt\nonumber
    \\
    I_4 := \int_{0}^{T}\Big{\{}  \left[ \gamma\big(s_{n}(t),t\big)s_{n}'(t)-\chi\big(
    s_{n}(t),t\big)\right]-\left[ \gamma\big(s(t),t\big)s'(t)-\chi\big(
      s(t),t\big)\right]\Big{\}}\Phi(s(t),t)\,dt\nonumber
    \\
    I_5:= \int_{0}^{T}\left[ g_n(t)-g(t) \right]\Phi(0,t)\,dt
  \end{gather}

  for arbitrary fixed $\Phi \in C^1(\bar{D})$.
  Each of the terms $I_{1},\ldots,I_5$ vanish as $n\to\infty$.
  For example, by CBS inequality
  \begin{gather}
    \lnorm{\int_{0}^{T}\int_{0}^{s(t)}\left( b_n - b \right) u_{n,x}\Phi
      \,dx \,dt} \leq \norm{b_n - b}_{L_2(D)} \norm{u_{n,x}}_{L_2(D)}
    \norm{\Phi}_{C(D)} \to 0~\text{as}~n \to \infty
  \end{gather}
  Which follows from uniform boundedness of $u_{n}\in B_{2}^{1,1}(D)$ and strong
  convergence of $b_n$ to $b$ in $L_2(D)$.
  The other two terms in $I_1$ are estimated in a similar way to show
  $\lnorm{I_1} \to 0$ as $n \to \infty$.
  Each term in $I_2$ is handled using CBS inequality as well:
  \begin{gather}
    \lnorm{\int_{0}^{T}\int_{s(t)}^{s_{n}(t)} au_{n,x}\Phi_{x} \,dx \,dt}
    \leq M \norm{\Phi_x}_{C(D)} \norm{s_n-s}_{C[0,T]}^{1/2} \norm{u_n}_{B_2^{1,0}(D)}\to 0~\text{as}~n \to \infty\nonumber
  \end{gather}
  Which follows from uniform boundedness of $u_{n}\in B_{2}^{1,0}(D)$ and uniform convergence of $s_{n}\to s$ on $[0,T]$.
  Treating each term in $I_2$ similarly, it follows that $\lnorm{I_{2}}\to 0$ as
  $n\to\infty$.
  Similarly, CBS inequality, continuity of the $L_2$ norm with respect to shift and uniform convergence of $s_n \to s$ imply $ \lnorm{I_{3}} \to 0$ and $\lnorm{I_4} \to 0$ as $n\to\infty$.
  Lastly, convergence of $g_{n}\to g$ strongly in $L_{2}(0,T)$ implies $\lnorm{I_{5}}\to 0$ as $n \to \infty$.

  Therefore, passing to the limit as $n\to\infty$ in~\eqref{eq:identity-for-limitpt-w} we see that the limit point $w$ satisfies
  \begin{equation*}
    0=\int_{0}^{T}\int_{0}^{s(t)}\big[ a  w_{x}\Phi_{x}-b w_{x}\Phi-c w \Phi + w_{t}\Phi \big]\, dx\,dt,
    \quad \forall \Phi \in C^1(\bar{D})
  \end{equation*}
  By extension of arbitrary $\Phi \in B_{2}^{1,1}(\Omega)$ to $B_2^{1,1}(D)$ and the density of $C^1(\bar{D})$ in $B_2^{1,1}(D)$, it follows that $w$ solves the Neumann problem~\eqref{eq:intro-pde}--\eqref{eq:intro-stefan} with $f=p=g=\gamma=\chi\equiv 0$.
  By the uniqueness of the solution to the Neumann problem
  it follows that $u_{n} \to u$ weakly in $B_{2}^{1,1}(D)$.
  By the Sobolev trace theorem~\cites{besov79,besov79a,nikolskii75}, CBS and Morrey inequalities it easily follows that
  \begin{equation*}
    \norm{u_{n}(x,T)-u(x,T)}_{L_{2}(0,s_{n}(T))}\to 0,\quad \norm{u_{n}(s_n(t),t)-u(s(t),t)}_{L_{2}(0,T)} \to 0~\text{as}~n\to\infty.
  \end{equation*}
  Therefore,
  \(
    \J(v)=\lim_{n\to\infty}\J(v_{n}) = J_{*}
  \)
  and $v \in V_{*}$.
  Theorem is proved.
\end{proof}
%
\begin{lemma}\label{lem:convergence-1}
  For $\epsilon>0$ define
  \(
    J_*(\pm \epsilon) = \inf_{V_{R\pm \epsilon}} \J(v).
    \)
    Then
  \(
    \lim_{\epsilon \to 0} J_*(\epsilon)
    = J_*
    = \lim_{\epsilon \to 0} J_*(-\epsilon)
  \)
\end{lemma}
Lemma~\ref{lem:convergence-1} is established as in~\cite{abdulla13}*{Lem.\ 3.9}
\begin{lemma}\label{lem:convergence-2}
  For $v \in V_R$,
  \(
    \lim_{n\to\infty}\fIn(\Q{}_n(v))
    = \J(v)
  \)
\end{lemma}
\begin{proof}
  Fix $v \in V_R$ and let $[v]_n = (\discreteControlVars) = \Q{}_n(v)$.
  Let $u = u(x,t;v)$ and $\big[ u([v]_n)\big]_n$ be the corresponding continuous
  and discrete state vector, respectively, and denote by $v^{n} =
    (\controlVarsWithN) = \P{}_{n}({[v]}_{n})$.
  By Sobolev embedding theorem, $s^{n}(t) \to s(t)$ uniformly on $[0,T]$.
  Let $\epsilon_m \downarrow 0$ be an arbitrary sequence, and define
  \[
    \Omega_m = \bk{(x,t):0 < x < s(t) - \epsilon_m, 0 < t \leq T}
  \]
  and fix $m>0$.

  In Theorem~\ref{lem:second-energy-est-conseq} it was shown that $\bk{\hat{u}^{\tau}}$
  converges to $u$ weakly in $B_2^{1,1}(\Omega_m)$ for any fixed $m$; by the
  embeddings of traces, it follows that $\bk{\hat{u}^{\tau}(s(t)-\epsilon_m,t)}$
  and $\bk{\hat{u}^\tau(x, T)}$ converge to the corresponding traces
  $u(s(t)-\epsilon_m, t)$ and $u(x,T)$ weakly in $L_2(0,T)$ and $L_2(0,
    s(t)-\epsilon_m)$, respectively.
  We shall prove that the corresponding traces of $u^\tau$ satisfy the same
  property.

  By Sobolev embedding theorem, it is enough to show that $\bk{u^\tau}$ and
  $\bk{\hat{u}^\tau}$ are equivalent in $B_2^{1,0}(\Omega_m)$.

  Denote by \(s_k^m = x_{\hat{\imath}}\) where
  \[
    \hat{\imath} =
    \max\bk{i \leq N:~
      -\epsilon_m \leq x_i - \max_{t_{k-1} \leq t \leq t_k} s(t) \leq - \frac{\epsilon_m}{2}}.
  \]
  Arguing as in~\cite{abdulla15}*{Eq.\ 101--104} it follows that there exists \(N=N(\epsilon_m)\) such that \(n>N\) implies
   \begin{equation}
    s_k^m < \min(s_k, s_{k-1}),~k=1,\ldots,n\label{eq:second-energy-est-conseq-smk-term-enough:disc}
  \end{equation}
  and accordingly
  \begin{gather}
    \norm{\D{\hat{u}^{\tau}}{x} - \D{u^{\tau}}{x}}_{L_2(\Omega_m)}^2
    = \frac{\tau^3}{3}\sum_{k=1}^{n} \sum_{i=0}^{\hat{\imath}-1}
      h_i u_{ix\bar{t}}^2(k)
    \leq \frac{\tau^3}{3}\sum_{k=1}^{n} \sum_{i=0}^{m_j-1}
      h_i \tilde{u}_{ix\bar{t}}^2(k)
     = O(\tau).
  \end{gather}
  Estimate the first term in $\fIn(\Q{}_n(v)) - \J(v)$ as
  \begin{gather}
    \lnorm{\beta_0\sum_{i=0}^{m_n-1} h_i \lnorm{u_i(n) - w_i}^2 \, dx - \beta_0\int_0^{s(T)} \lnorm{u(x,T) - w(x)}^2 \, dx}\nonumber
    \\
    \leq \beta_0 \bigg{\{}\lnorm{\sum_{i=0}^{\hat{\imath}-1} \left[ h_i
    \lnorm{u_i(n) - w_i}^2 - \int_{x_i}^{x_{i+1}}\lnorm{u(x,T) - w(x)}^2\, dx \right]}
    + I_{n,m} + \tilde{I}_{m}\bigg{\}}\label{eq:In-Qn-v-term1-1}
    \intertext{where}
    I_{n,m}
    =\lnorm{\sum_{\hat{\imath}}^{m_n - 1} h_i \lnorm{u_i(n) - w_i}^2},\quad \tilde{I}_{m}
    = \lnorm{\int_{s_n^m}^{s(T)} \lnorm{u(x,T) - w(x)}^2 \, dx}
    \end{gather}
    By absolute continuity of the integral, $\tilde{I}_m \to 0$ as $m \to \infty$.
    Considering $I_{n,m}$,
    \[
      I_{n,m}
      \leq 2 \lnorm{\sum_{\hat{\imath}}^{m_n - 1} h_i \lnorm{u_i(n)}^2} + \lnorm{\sum_{\hat{\imath}}^{m_n - 1} h_i\lnorm{w_i}^2}
    \]
    By Morrey's inequality,
  \begin{equation*}
    \lnorm{\sum_{\hat{\imath}}^{m_n - 1} h_i \lnorm{u_i(n)}^2}
    \leq C \lnorm{s^n(T)-s(T)+\epsilon_m}
    \norm{\hat{u}(x;n)}_{B_2^1(0,\l)}^2
  \end{equation*}
  From~\eqref{eq:thm:first-energy-est-2} and~\eqref{eq:thm:second-energy-est}, it follows that
  \begin{equation}
    \norm{\hat{u}(x;n)}_{B_2^1(0,\l)}^2 \leq C_1\label{eq:u-x-n-uniform-bound}
  \end{equation}
  For a constant $C_1$ depending on the given data $\phi$, $f$, etc.\ but not $\tau$ (or $m$).
  Now, considering the second term in $I_{n,m}$, by CBS inequality,
  \begin{gather*}
    \lnorm{\sum_{\hat{\imath}}^{m_n - 1} h_i \lnorm{w_i}^2}
    = \lnorm{\sum_{\hat{\imath}}^{m_n - 1} \frac{1}{h_i} \lnorm{\int_{x_i}^{x_{i+1}}w(x)
    \,dx}^2}
    \leq \lnorm{\int_{s(T)}^{s^n(T)} \lnorm{w(x)}^2 \,dx}
    +\lnorm{\int_{s(T)-\epsilon_m}^{s(T)} \lnorm{w(x)}^2 \,dx}
  \end{gather*}
  By absolute continuity of the integral and convergence $s^n(T) \to s(T)$, it
  follows that there is some $N_1=N_1(m)$ such that for $n>N_1$,
  \begin{equation}
    \lnorm{\sum_{\hat{\imath}}^{m_n - 1} h_i \lnorm{w_i}^2}
    \leq 2 \int_{s(T)-\epsilon_m}^{s(T)}  \lnorm{w(x)}^2 \,dx
    + \frac{1}{m}\label{eq:conv-Inm-term2-bound}
  \end{equation}
  By~\eqref{eq:u-x-n-uniform-bound} and~\eqref{eq:conv-Inm-term2-bound}, it follows that for $n>N_1$
  \begin{gather}
    0
    \leq I_{n,m}
    \leq C C_1 \left( \epsilon_m + \lnorm{s^n(T)-s(T)} \right)
    + 2 \int_{s(T)-\epsilon_m}^{s(T)} \lnorm{w(x)}^2 \,dx + \frac{1}{m}\label{eq:conv-Inm-bound}
  \end{gather}
  By~\eqref{eq:In-Qn-v-term1-1} and~\eqref{eq:conv-Inm-bound}, it follows that
  \begin{gather*}
    0 \leq \limsup_{n\to\infty}\lnorm{\beta_0\sum_{i=0}^{m_n-1} h_i \lnorm{u_i(n) - w_i}^2 \, dx - \beta_0\int_0^{s(T)} \lnorm{u(x,T) - w(x)}^2 \, dx} \nonumber
    \\
    \leq C C_1\epsilon_m
    + 2 \int_{s(T)-\epsilon_m}^{s(T)}  \lnorm{w(x)}^2 \,dx + \frac{1}{m} + \tilde{I}_m
  \end{gather*}
  for all $m$.
  Passing to the limit as $m \to \infty$ it follows that
  \begin{equation*}
    \lim_{n\to\infty}\beta_0\sum_{i=0}^{m_n-1} h_i \lnorm{u_i(n) - w_i}^2  = \beta_0\int_0^{s(T)} \lnorm{u(x,T) - w(x)}^2 \, dx
  \end{equation*}
The convergence of the second and third terms of $\fIn$ to corresponding terms in $\J$ is established in a similar way.
Lemma is proved.
\end{proof}
\begin{lemma}\label{lem:convergence-3}
  For arbitrary $[v]_n \in \discreteControlSet$,
  \(
    \lim_{n\to\infty} \left(\J\big(\P_n([v]_n)\big) - \fIn\big([v]_n \big) \right)
    = 0
  \)
\end{lemma}
\begin{proof}
  Let $[v]_n \in \discreteControlSet{}$ and $v^n = (\controlVarsWithN{}) = \P_n([v]_n)$.
  Then $\bk{\P_n([v]_n)}$ is weakly precompact in $\controlSpace{}$; assume that the whole sequence converges to $\tilde{v}=(\controlVarsWithTilde{})$.
  Then $\tilde{v} \in V_R$, and moreover, Rellich-Kondrachov compactness theorem implies that $(\controlVarsStronglyConvergeWithN{}) \to (\controlVarsStronglyConvergeWithTilde{})$ strongly in $\controlSpaceStronglyConverge{}$; in particular, $s^n \to \tilde{s}$ uniformly on $[0,T]$.
  Write the difference $\J\big(\P_n([v]_n)\big) - \fIn\big([v]_n \big)$ in the preceding notation, as
  \begin{equation*}
    \fIn\big([v]_n \big) - \J\big(\P_n([v]_n)\big)
    = \fIn\big([v]_n \big) - \J\big(v^n\big)
    = \fIn\big([v]_n \big) - \J(\tilde{v}) + \J(\tilde{v}) - \J\big(v^n\big)
  \end{equation*}
  By weak continuity of $\J$, we have
  \(
    \lim_{n\to\infty} \left(\J(\tilde{v}) - \J\big(v^n\big)\right) = 0.
  \)
  It remains to be shown that
  \begin{equation*}
    \lim_{n\to\infty} \left(\fIn\big([v]_n \big) - \J(\tilde{v})\right) = 0
  \end{equation*}
  Since $\tilde{v} \in V_{R+\epsilon}$ for some $\epsilon > 0$, and by strong convergence of $\P_n({[v]}_n) \to \tilde{v}$, a nearly identical argument to the proof of Lemma~\ref{lem:convergence-2} establishes this result.
\end{proof}
By Lemmas~\ref{lem:convergence-1}--\ref{lem:convergence-3}
and~\cite{abdulla15}*{Lem.\ 2.2}, Theorem~\ref{thm:convergence-discrete} is proved.
\section{Conclusions}\label{sec:conclusions}
Motivated by the new variational formulation of the inverse Stefan problem
and by applying the methods developed in~\cites{abdulla13,abdulla15}, identification of coefficients, heat flux, and density
of heat sources in the second order parabolic free boundary problem arising in biomedical problem on the laser ablation of tissues is analyzed in a Besov spaces framework in this paper.

The main idea of the new variational formulation is an optimal control setting,
where the free boundary, coefficients, heat flux, and heat sources are components
of the control vector.
Discretization of the variational formulation is pursued using the method of finite differences,
and convergence of the discrete optimal control problems with respect to functional and
control is proven.

This creates a rigorous basis for the development of an iterative gradient type numerical method
of low computational cost, and allows for the regularization of the error existing in the
information on the phase transition temperature and other experimental measurements.
%

\end{document}